\documentclass[oneside,12pt]{amsart}
\usepackage{amscd,amsmath,latexsym}
\usepackage[mathcal]{euscript}
\usepackage{bm,pdfpages,verbatim}
\newtheorem{thm}{Theorem}[section]
\newtheorem{cor}[thm]{Corollary}

\newtheorem{prop}[thm]{Proposition}
\theoremstyle{definition}

\theoremstyle{remark}

\def\nd{\noindent}

\begin{document}

\author[William Browder]{William Browder}
\address{Department of Mathematics,
Princeton  University, Princeton, NJ 08544, U.S.A.}
\email{Bill Browder <browder@math.princeton.edu>}

\

\title[Spheroidal groups, virtual cohomology and lower dimensional $G$-spaces]
{Spheroidal groups, virtual cohomology and lower dimensional $G$-spaces}

\



\maketitle

\centerline{{\bf This paper is dedicated to the memory of my beloved friend and colleague}}
\centerline{{\bf  Sam Gitler.}}

\begin{abstract}
A space is defined to be ``$n$-spheroidal''  if it has the homotopy type of an $n$-dimensional CW-complex $X$
with $H_{n}(X; \mathbb{Z})$ not zero and finitely generated. A group $G$ is called ``$n$-spheroidal''  if its classifying space 
$K(G,1)$ is $n$-spheroidal. Examples include fundamental groups of compact manifold $K(G,1)$'s. Moreover,
the class of groups $G$ which are $n$-spheroidal  for some $n$, is closed under products, free products, and group extensions.
If $Y$ is a space with $\pi_{1}(Y)$ $n$-spheroidal, and if $H_{k}(Y;\mathbb{F}_{p})$ is non-zero and finitely generated,
and if $H_{i}(Y;\mathbb{F}_{p}) = 0$ for $i>k$, then $H_{n+k}(\overline{Y};\mathbb{F}_{p}) \neq 0$ for  
$\overline{Y}$ a finite sheeted covering space of $Y$. Hence, dim$(Y)  \geq n+k$. Thus, it follows that if dim$(Y) < n$,
and if $H_{k}(Y;\mathbb{F}_{p}) \neq 0$ and $H_{i}(Y;\mathbb{F}_{p}) = 0$ for $i>k>0$, then $H_{k}(Y;\mathbb{F}_{p})$
is not finitely generated. Similar results follow  for $Y\subset K(G,1)$.
 \end{abstract}

\section{Introduction}
We Êcall Ê Êan Ê Ê$n$-dimensional Ê ÊCW Ê Êcomplex Ê Ê$X$ Ê``$n$-spheroidal'' Ê Êif  the integral homologyÊ $H_n(X)$ Êis Ê
(a) Ênon Êtrivial Ê and Ê(b) Êfinitely Êgenerated, Ê(``weakly Ê$n$-spheroidal'' Ê Êwithout Ê(b)). Ê Note that this implies that $H_n(X)$ is free abelian, being isomorphic to the cycles in $C_n(X)$, and this holds for any finite sheeted covering of $X$. Also every homology class in  $H_n(X)$ is detected by a map of  $X$ into $S^n$, which is the origin of this choice of notation. Similar statements hold Êfor Ê Êchain Êcomplexes, Êand  for Êgroups Êby Êapplying Êthis Êcondition Êto Êthe Êclassifying Êspace Êof Ê Ê$G$.

\medskip
Let $p$ be prime and $\mathbb{F}_{p}$ denote the field with $p$ elements.

\medskip
\begin{thm}Ê ÊLet Ê Ê$X$ Ê Êbe Êa ÊCW Êcomplex Êwith weakly Ê Ê$n$-spheroidal Ê Êfundamental Ê group Êand Êwith Êuniversal Êcovering Ê Êspace Ê$Y$  Ê Êsuch Êthat Ê Ê$H_k(Y;\mathbb{F}_{p})$ Ê Êis Ênon Ê zero Êand Êfinitely Êgenerated Êfor Êsome Êprime Ê Ê$p$, Êand Êthat Ê Ê$H_q(Y;\mathbb{F}_{p}) Ê Ê= Ê Ê0$ Ê for Ê$q Ê Ê> Ê Êk$. Ê ÊThen Ê Êthe Êdimension Êof Ê$Y$  Ê Êis greater Êthan Êor Êequal Êto Ê Ê$n Ê+ Êk$. Ê (For Êexample Êif Êan Ê$n$-spheroidal Êgroup Ê$G$ Êacts Êfreely Êon Ê Ê$M Ê\times ÊR^k$ Ê, Êwhere Ê$M$ Ê Êis Êa Êcompact Êmanifold, then Ê $k$ Êis greater Êthan Êor Êequal Êto Ê Ê$n$). Ê
\end{thm}

\medskip
This Êimplies Êthat Ê Êfor Êa Êfree Ê Ê$G$ ÊCW Êcomplex Ê Ê$Y$, Ê $G$ Ê Ê$n$-spheroidal Êand Ê$\dim Y Ê< Ên$, Ê Êthe Êtop Êhomology Ê with Ê Ê$\mathbb{F}_{p}$ Ê Êcoefficients Êof Ê Ê$Y$  Ê Êis Ênot Êfinitely Êgenerated Êfor Êany Ê Êprime Ê Ê$p$. 

\medskip
I thank the referee for comments, and particularly for calling my attention to the two papers of N.~Petrosyan which
contain related results \cite[Propsition $1.3$]{p1} and \cite[Theorem $4.2$]{p2} which are proved in similar way. I
thank J.~Koll\'{a}r for drawing my attention to such problems (see Section \ref{sec:results}).

\section{Spheroidal Ê Êgroups}

For Êany Êgroup Ê Ê$G$ Ê Êone Êcan Êconstruct Êa Ê ÊCW Ê Êcomplex Ê Ê$K(G,1)$ Ê (called Êthe Êclassifying Êspace Êof Ê Ê$G$) Ê Êwith fundamental Êgroup Ê Ê$G$ Ê Êand Êall Êother Êhomotopy Êgroups Êtrivial, Ê (see Êfor Êexample Ê Ê[\cite{kb}). Ê ÊThe Êhomotopy Êtype Êof Ê Ê$K(G,1)$ Ê Êdepends Ê only Êon Ê Ê$G$, Ê Êand Êits Êcohomology Ê Êis Êcalled Êthe Êcohomology Êof Êthe Êgroup. Ê (An Êalgebraic Êdefinition Êusing Êprojective Êresolutions Êis Êalso Ê standard.) Ê

\medskip
We Êcall Ê Êa Êgroup Ê Ê$G$ Ê Ê$n$-spheroidal Ê Êif Ê Ê$K(G,1)$ Ê Êis Ê $Ên$-spheroidal Êup Êto Êhomotopy Ê (similarly Êfor Ê Êweakly Ê$n$-spheroidal). 

\begin{prop} Let Ê$f:G \toÊL$ Êbe Êa Êhomomorphism Êof Êgroups Êwhere Ê Ê$L$ Êis Ê Êfinite Êand Êlet Ê$D$  Êbe Êthe Êkernel Êof Ê 
$f$, Êso Êthat Ê Ê$D$ Êhas Êfinite Êindex Êin $G$. ÊIf $G$ Ê Êis Ê$n$-spheroidal the so isÊ Ê$D$Ê (similarly Êfor Êweakly Ê$n$-spheroidal). Ê The Êconverse Êholds Êif Ê Ê$K(G,1)$ Êis Ê $Ên$-dimensional. Ê
\end{prop}

\begin{proof}
We use Êthe Êspectral Êsequence Êfor Ê Ê
$$K(D,1) \longrightarrow K(G,1) \longrightarrowÊK(L,1).$$ Ê 

\nd  To Ê verify Êthat Ê$H_n(D)$ Êis Ênon-zero, Êand  Êis Êfinitely Êgenerated Êif Ê$H_n(G)$ Êis, we use the fact that $L$ is finite,  so that the reduced  homology $\overline{H}_{\ast}\big(K(L,1)\big)$ is a $\mathbb{Z}$-torsion module,Ê (annihilated by $|L|$). ÊIt Êis Êclear that Ê Ê$K(D,1)$ Êis Ê
$n$-dimensional Êsince Êit Êis Êthe Êquotient Ê Êof Êthe Ê($n$-dimensional) universal Ê space Êfor Ê Ê$G$ Ê Êby Êthe Êsubgroup Ê Ê$D$. Ê\end{proof}

\medskip
The Êconverse Êfollows Êeasily Êwith Êthe Êextra Êcondition. Ê

\medskip
A Ê ÊCW Êcomplex Ê Ê$X$ Ê Êis Êsimply Ê Êcalled Ê Ê``spheroidal''Ê if Êit Êis Ê
$n$-spheroidal Ê Êfor Êsome Ê Ê$n$. Ê
If Ê Ê$K(G,1)$ Ê Êhas Êthe Êhomotopy Êtype Êof Êa Êclosed Êoriented Ê$n$-manifold, Ê it Êis Ê$n$-spheroidal, Ê Êand Êfinitely Êgenerated Êfree Êgroups Êare Ê 1-spheroidal. Ê ÊThe Êclass Êof Ê Êspheroidal Êgroups Êis Êclosed Êunder Ê finite Êdirect Êproducts Êand Êfinite Ê Êfree Êproducts, Ê Êwhich Êshows Êthat Ê there Êare Êmany Ê Êsuch Êgroups. ÊIt Êwill Êfollow Êfrom Êour Êarguments Êbelow, Ê that Êextensions Êof (weakly) Ê$n$-spheroidal Êgroups Êby $m$-spheroidal Êgroups Ê Êare Ê(weakly) Ê Ê$(n+m)$-spheroidal, Ê Êand extensions Ê
of Ê Ê$n$-spheroidal Êgroups Êby Ê Ê$m$-spheroidal Êgroups Êare Ê $(n+m)$-spheroidal. 

\medskip
Thus Êthe Êclass Êof Ê(weakly) Ê$n$-spheroidal Êgroups Êis Êlarge. Ê

\section{The results}\label{sec:results}
\begin{thm}\label{thm:A}Ê ÊLet Êthe Êgroup Ê Ê$G$ Ê Êbe Êweakly Ê$n$-spheroidal, Ê Êand Êlet Ê Ê$Y$  Êbe Ê Êa Ê projective Ê Ê
$\mathbb{Z}G$-chain Êcomplex Êsuch Êthat Ê Ê$Y_k Ê Ê= Ê Ê0$ Ê Êfor Ê $Êk Ê> Ên$ Ê(for example, the chains of Ê a Ê Ê$G$-free ÊÊCW Êcomplex Êof Êdimension Ê Ê$\leq n$).
ÊIf Ê$j>0$ Êis Êmaximal Ê such Êthat Ê Ê$H_j(Y;\mathbb{F}_{p})$ Ê Êis Ênon-zero, Êand Êif Ê Ê$H_j(Y;\mathbb{F}_{p})$ Êis Êfinitely Ê generated, then Ê$H_{n+j}(Y/D;\mathbb{F}_{p})$ Ê Êis Ênot Êzero Êfor Êsome Êsubgroup Êof Ê finite Êindex Ê $ÊD$ Êof Ê Ê Ê$G$. Ê
\end{thm}

\medskip
\begin{cor}\label{cor:B}ÊIf Ê$ ÊG$ Ê Êis Êweakly Ê Ê$n$-spheroidal Êand Ê Êacts Êfreely Êon Ê $ÊY$  Ê Êof Ê dimension ${}Ê Ê< Ê Ên$, Ê Êthen Êthe Êtop Ênon-zero Êdimensional Ê Ê$H_j(Y;\mathbb{F}_{p})$ Ê Êis Ênot Ê finitely Êgenerated Êfor Êevery Êprime Ê$p$.
\end{cor}

\medskip
\nd {\bf Question}: Ê ÊWith Êthe Êhypothesis Êof Ê ÊCorollary \ref{cor:B}, Êcan Êone Êconclude Êthat Êfor Êevery Ê$m Ê> Ê0$,
Ê$H_m(Y)$ Ê non-zero Êimplies $ÊH_m(Y)$ Êis Ênot Êfinitely Êgenerated? Ê

\medskip
\begin{thm}\label{thm:C}ÊLet Ê$X$ Êbe Êan Ê Ê$n$-spheroidal Ê Êspace, Ê Ê$f: ÊE Ê\to ÊX$ Ê Êbe Êa Êfibre Ê space Êwith Êfibre Ê Ê$Y$. Ê ÊIf Êfor Êsome Êprime Ê Ê$p$, Ê$H_i(Y;\mathbb{F}_{p})$ Êis Êfinitely Ê generated Êand Ênon Êzero, Êand Ê Ê Ê Ê$H_k(Y;\mathbb{F}_{p})$ Êis Ê Êzero Êfor $Êk Ê> Êi$, Êthen Ê $H_{n+i}(E';\mathbb{F}_{p})$ Ê Êis Ênon Êzero Ê Êfor Êsome Êpullback Ê$E'$ Êof $ÊE$ Ê over $ÊX'$ Êa Êfinite Êsheeted Êcovering Êof Ê Ê$X$, Êcorresponding Êto  some subgroup Ê$D$ Êof Êfinite Ê index Êin Ê$G Ê= Ê\pi_1(X)$.
\end{thm} Ê
\begin{proof}Ê ÊThe Ê Êaction Ê Êof Ê Ê$G$ Ê Êon Êthe Êfibre Ê Ê$Y$  Êinduces Êan Êaction Êof Ê $G$ Ê Êon Ê Ê$H_{\star}(Y;\mathbb{F}_{p})$, in other Êwords, Êa Êhomomorphism Ê Êof Ê Ê$G$ Ê Êto Ê Aut$(H_i(Y ;Ê\mathbb{F}_{p}))$. The latter is a finite group because it Êis Êa Ê Êlinear Êgroup Ê Êof Ê Êa Êfinite Êdimensional Ê vector Êspace Êover Ê Ê$\mathbb{F}_{p}$. ÊIf Ê Ê$D$ Ê Êis Êthe Ê kernel Êof Êthis Êhomomorphism, then Ê Ê$D$ Ê Êacts Êas Êthe Êidentity Ê Êon
$H_i(Y;\mathbb{F}_{p})$. Ê Note Êthat Ê Ê$H_n(X)$ Êis Êa Êfree Êabelian Êgroup Êfor Êan Ê$n$ Êdimensional ÊCW Ê complex Ê Ê$X$. ÊConsider Êthe Êpullback Ê Ê$E'\to ÊX'$, Ê Êwhere Ê Ê$X'$ Ê Êis Êthe Ê covering Êof Ê Ê$X$ Ê Êcorresponding Êto Êthe Êsubgroup Ê Ê$D$ Ê Êof Ê Ê$G$. ÊThen Êthe Êhighest Êdimensional Êgroup Êin Ê Ê$E^2$ Êof Êthe Êspectral Ê sequence Êfor Ê Ê$E'\to ÊX'$ Ê Êis  $$ E^2_{n,i} Ê Ê= Ê H_n(X'; ÊH_i(Y ;\mathbb{F}_{p})) Ê= ÊH_n(X')\otimes ÊH_i(Y ;\mathbb{F}_{p})$$
which is Ê non Êzero Êsince Ê$H_n(X')$ Êis Êfree Êand Êthe Êaction Êof Êthe Êfundamental Êgroup Ê on Êthe Êhomology Êof Êthe Êfibre Êis Êtrivial. Ê ÊIt Êis Êin Êthe Êkernel Êof Êevery Ê differential Ê Êsince $Ê ÊH_s(Y ;\mathbb{F}_{p}) Ê Ê= Ê Ê0$ Ê Êfor Ê $Ês Ê> Êi$, Ê Ê Ê Êand Êis Ênever Ê Êa Ê boundary Ê Êsince Ê Ê$H_t(X;\mathbb{F}_{p}) Ê Ê= Ê Ê0$ Ê Êfor Ê$ Êt Ê Ê> Ên$. Ê Ê(``The Êupper Êright Ê hand Êcorner'' Êargument). Ê ÊHence Ê Ê$E^2_{n,i} Ê= Ê ÊE^\infty_{n,i}$, Êand Ê  Ê Ê$H_{n+i}(E';\mathbb{F}_{p})$ Ê Êis Êtherefore Ênon Êzero. Ê
\end{proof}

\medskip
\nd Theorem Ê\ref{thm:A}Ê  follows by Êsetting Ê Ê$X Ê Ê= Ê ÊK(G,1)$ Êetc. Ê

\medskip
\begin{cor}\label{cor:D} ÊLet Ê Ê$G$ Ê Êbe Ê Êan Ê$n$-spheroidal Êgroup. Ê ÊIf Ê Ê$Y$  Ê Êis Êa Ê Ê$G$ Ê Êsubset Ê of Êthe Êuniversal Ê Ê$G$ Ê Êspace Ê(i.e., Êthe Êuniversal Êcover Êof Ê$ ÊK(G,1)$) Êand ifÊ Êthe Ê highest Ê Êdimensional Ênon-zero Êcohomology Êgroup Ê Ê$H^k(Y)$ Êis Êfinitely Ê Êgenerated, Ê Êthen Ê Ê$k Ê  Ê Ê Ê Ê= Ê Ê0$ Ê Êand Ê Ê$Y$  Ê Êis Ê Êacyclic Ê Êso Êthat $H^*(Y/G)$ Ê Êis Êisomorphic Ê Êto Ê$H^*(K(G,1))$. Ê
\end{cor}
 
 \medskip
\begin{cor}\label{cor:E}Ê ÊWith Êthe Êhypotheses Êof Ê ÊCorollary \ref{cor:D}, Ê Êsuppose Êin Êaddition Ê Ê$K(G,1)$ Ê Êis Êminimal Ê(for Êexample Êa Êcompact Ê $n$-manifold), Ê Êthen Ê Ê$Y$  Ê Êis Êthe Êuniversal Ê$G$ Êspace. Ê
\end{cor}

\medskip
\begin{proof} If $Y/G$ Ê Êis Êa Êsubspace Êof Êa Êcompact Êmanifold Ê Ê$K(G,1)$ Êwith Êthe Êsame Êhomology, Êthen
$Y/G Ê Ê= Ê ÊK(G,1)$ Êand Ê$Y$  Êis Êthe Êuniversal Ê Ê$G$Ê space. Ê Any Êsubset Êof Êa Êcompact Êmanifold, having Êthe Êsame Êtop Êdimensional Êhomology, must be the Êwhole Êmanifold. Ê
\end{proof}

\medskip
Corollary \ref{cor:E} answers  in the affirmative the following question Êput Êto Ême Êby ÊJ.~Koll\'{a}r Êwhich Êstimulated Ê Êthis Êstudy, \cite{kp}:

\medskip
\nd {\bf Koll\'{a}r's Question:} Let $f\colon \mathbb{R}^{n}\longrightarrow (S^{1})^{n} $ be the universal cover. 
Suppose $X$ is a subset of 
$(S^{1})^{n} $ such that $f^{-1}(X) \subseteq \mathbb{R}^{n}$ has the homotopy type of a finite complex; is $X$
equal to $\mathbb{R}^{n}$?

\bibliographystyle{amsalpha}

\end{document}